\documentclass{amsart}
\usepackage{amsmath,amssymb,amsthm}
\usepackage{bm}
\usepackage{extarrows}
\usepackage{mathrsfs}
\usepackage{stmaryrd}
\usepackage[colorlinks=true]{hyperref}
\usepackage[all]{xy}

\newtheorem{theorem}{Theorem}[section]

\newtheorem{lemma}[theorem]{Lemma}
\newtheorem{proposition}[theorem]{Proposition}
\theoremstyle{definition}

\theoremstyle{remark}
\newtheorem{remark}[theorem]{Remark}
\numberwithin{equation}{section}

\newcommand\ra{\rightarrow}

\newcommand\wt{\widetilde}
\newcommand\half{\frac{1}{2}}
\newcommand\mmid{\parallel}
\newcommand\simto{\stackrel{\sim}{\longrightarrow}}
\font\cyr=wncyr10
\newcommand\Sha{\hbox{\cyr X}}
\newcommand\set[1]{{\left\{#1\right\}}}
\newcommand\pair[1]{\langle{#1}\rangle}
\newcommand\svec[2]{\begin{pmatrix}#1\\#2\end{pmatrix}}
\newcommand\leg[2]{\left(\frac{{#1}}{#2}\right)}
\newcommand\legbigg[2]{\biggl(\frac{{#1}}{#2}\biggr)}
\newcommand\aleg[2]{\left[\frac{{#1}}{#2}\right]}
\newcommand\rmT{{\mathrm{T}}}

\newcommand\bfd{{\mathbf{d}}}
\newcommand\bfx{{\mathbf{x}}}
\newcommand\bfy{{\mathbf{y}}}
\newcommand\bfz{{\mathbf{z}}}
\newcommand\bfA{{\mathbf{A}}}
\newcommand\bfD{{\mathbf{D}}}
\newcommand\bfI{{\mathbf{I}}}
\newcommand\bfM{{\mathbf{M}}}

\newcommand\bfO{{\mathbf{O}}}

\newcommand\BA{{\mathbb{A}}}
\newcommand\BF{{\mathbb{F}}}

\newcommand\BQ{{\mathbb{Q}}}
\newcommand\BR{{\mathbb{R}}}
\newcommand\BZ{{\mathbb{Z}}}

\providecommand\CD{{\mathcal{D}}}
\newcommand\CE{{\mathcal{E}}}
\newcommand\CH{{\mathcal{H}}}
\newcommand\CL{{\mathcal{L}}}

\newcommand\CP{{\mathcal{P}}}
\newcommand\CQ{{\mathcal{Q}}}
\newcommand\CU{{\mathcal{U}}}
\newcommand\CV{{\mathcal{V}}}

\newcommand\mse{\mathscr{E}}

\newcommand\Cl{{\mathrm{Cl}}}

\newcommand\diag{{\mathrm{diag}}}
\DeclareMathOperator\Ker{Ker}
\DeclareMathOperator\rank{rank}
\newcommand\Sel{{\mathrm{Sel}}}
\newcommand\sgn{{\mathrm{sgn}}}
\newcommand\tors{{\mathrm{tors}}}

\title[On a comparison of Cassels pairings]{On a comparison of Cassels pairings of different elliptic curves}
\author{Shenxing Zhang}
\address{School of Mathematics, Hefei University of Technology, Hefei, Anhui 230000, China}
\email{zhangshenxing@hfut.edu.cn}
\date{\today}

\begin{document}
\maketitle

\begin{abstract}
Let $e_1,e_2,e_3$ be nonzero integers satisfying $e_1+e_2+e_3=0$.
Let $(a,b,c)$ be a primitive triple of odd integers satisfying $e_1a^2+e_2b^2+e_3c^2=0$.
Denote by $E: y^2=x(x-e_1)(x+e_2)$ and $\mathcal E: y^2=x(x-e_1a^2)(x+e_2b^2)$.
Assume that the $2$-Selmer groups of $E$ and $\mathcal E$ are minimal.
Let $n$ be a positive square-free odd integer, where the prime factors of $n$ are nonzero quadratic residues modulo each odd prime factor of $e_1e_2e_3abc$.
Then under certain conditions, the $2$-Selmer group and the Cassels pairing of the quadratic twist $E^{(n)}$ coincide with those of $\mathcal E^{(n)}$.
As a corollary, $E^{(n)}$ has Mordell-Weil rank zero without order $4$ element in its Shafarevich-Tate group, if and only if these holds for $\mathcal E^{(n)}$.
We also give some applications for the congruent elliptic curve.
\end{abstract}

\section{Introduction}
The quadratic twists family of a given elliptic curve are studied in many articles.
What we want to study is when two different families have similar arithmetic properties.
In \cite{ParshinZarhin1989}, given abelian varieties $A_1$ and $A_2$ over $K$ whose ranks agree over each finite extension of $K$, Zarhin asks if $A_1$ is necessarily isogenous to $A_2$.
In \cite{MazurRubin2015}, Mazur and Rubin consider the Selmer groups instead of ranks.
They give a sufficient condition on when the Selmer ranks of elliptic curves $E_1$ and $E_2$ agree over at most quadratic extension of $K$.
In particular, there are non-isogenous $p^k$-Selmer companions.
It's also known that if the $\ell$-Selmer ranks of $E_1$ and $E_2$ agree over each finite extension of $K$ for all but finitely many primes $\ell$, then $E_1$ and $E_2$ are $K$-isogenous, see~\cite{Chiu2020}.
For related results, see also~\cite{Kisilevsky2004, Yu2019}.

In this paper, we will study when the ranks of elliptic curves with full $2$-torsion agree over a set of quadratic fields.
More precisely, let
\[E=\mse_{e_1,e_2}:y^2=x(x-e_1)(x+e_2)\]
be an elliptic curve defined over $\BQ$ with full $2$-torsion, where $e_1,e_2,e_3=-e_1-e_2$ are non-zero integers.
Let $E^{(n)}=\mse_{e_1n,e_2n}$ be a quadratic twist of $E$, where $n$ is an odd positive square-free integer.
Let $(a,b,c)$ be a primitive triple of odd integers satisfying
\[e_1a^2+e_2b^2+e_3c^2=0.\]
Denote by $\CE=\mse_{e_1a^2,e_2b^2}$ and $\CE^{(n)}=\mse_{e_1na^2,e_2nb^2}$ its quadratic twist.

Since we want to compare $E^{(n)}$ for different triples $(e_1,e_2,e_3)$, we will assume that 
\[\gcd(e_1,e_2,e_3)=1\text{ or }2\]
for simplicity.
By a translation of $x$, one can show that $E\cong\mse_{e_2,e_3}\cong\mse_{e_3,e_1}$.
This gives a symmetry on $(e_1,e_2,e_3)$.
Without loss of generality, we may assume that $v_2(e_3)$ is maximal among $v_2(e_i)$, where $v_2$ is the normalized $2$-adic valuation.
Then $v_2(e_1)=v_2(e_2)<v_2(e_3)$.
We will write $2^{v_2(x)}\mmid x$.

Denote by $\Sel_2(E/\BQ)$ the $2$-Selmer group of $E$.
Then we have an exact sequence
\[0\ra E(\BQ)/2E(\BQ)\ra \Sel_2(E/\BQ)\ra \Sha(E/\BQ)[2]\ra0.\]
If $E$ has no rational point of order $4$, then $E(\BQ)[2^\infty]\cong(\BZ/2\BZ)^2$ since it has full $2$-torsion.
Therefore, $\Sel_2(E/\BQ)$ contains $E(\BQ)[2]\cong(\BZ/2\BZ)^2$.

The following theorems generalize the observations in \cite{WangZhang2022}, which give a relation between $E^{(n)}$ and $\CE^{(n)}$.

\begin{theorem}[= Theorem~\ref{thm:sha-odd}]\label{thm:main-odd}
Let $n$ be an odd positive square-free integer coprime with $e_1e_2e_3abc$, whose prime factors are quadratic residues modulo each odd prime factor of $e_1e_2e_3abc$.
Assume that
\begin{itemize}
\item $e_1,e_2$ are odd and $2\mmid e_3$.
\end{itemize}
If $\Sel_2(E/\BQ)\cong\Sel_2(\CE/\BQ)\cong(\BZ/2\BZ)^2$, then the following are equivalent:
\begin{enumerate}
\item $\rank_\BZ E^{(n)}(\BQ)=0$ and $\Sha\bigl(E^{(n)}/\BQ\bigr)[2^\infty]\cong(\BZ/2\BZ)^{2t}$;
\item $\rank_\BZ \CE^{(n)}(\BQ)=0$ and $\Sha\bigl(\CE^{(n)}/\BQ\bigr)[2^\infty]\cong(\BZ/2\BZ)^{2t}$.
\end{enumerate}
\end{theorem}

When $\gcd(e_1,e_2,e_3)=2$, $E^{(n)}=\mse_{e_1/2,e_2/2}^{(2n)}$ is an even quadratic twist of an elliptic curve in Theorem~\ref{thm:main-odd}. In which case, an additional condition is required.
\begin{theorem}[= Theorem~\ref{thm:sha-even}]\label{thm:main-even}
Let $n$ be an odd positive square-free integer coprime with $e_1e_2e_3abc$, whose prime factors are quadratic residues modulo each odd prime factor of $e_1e_2e_3abc$.
Assume that
\begin{itemize}
\item $2\mmid e_1, 2\mmid e_2, 4\mid e_3$;
\item both $E$ and $E^{(n)}$ have no rational point of order $4$;
\item if $e_2>0$ and $e_3<0$, then every prime factor of $n$ is congruent to $1$ modulo $4$, or every odd prime factor of $e_2e_3bc$ is congruent to $1$ modulo $4$;
\item if $e_3>0$ and $e_1<0$, then every prime factor of $n$ is congruent to $1$ modulo $4$, or every odd prime factor of $e_1e_3ac$ is congruent to $1$ modulo $4$;
\item if $e_1>0$ and $e_2<0$, then every prime factor of $n$ is congruent to $1$ modulo $4$.
\end{itemize}
If $\Sel_2(E/\BQ)\cong\Sel_2(\CE/\BQ)\cong(\BZ/2\BZ)^2$, then the following are equivalent:
\begin{enumerate}
\item $\rank_\BZ E^{(n)}(\BQ)=0$ and $\Sha\bigl(E^{(n)}/\BQ\bigr)[2^\infty]\cong(\BZ/2\BZ)^{2t}$;
\item $\rank_\BZ \CE^{(n)}(\BQ)=0$ and $\Sha\bigl(\CE^{(n)}/\BQ\bigr)[2^\infty]\cong(\BZ/2\BZ)^{2t}$.
\end{enumerate}
\end{theorem}

For general triples $(e_1,e_2,e_3)$, we require that the prime factors of $n$ are congruent to $1$ modulo $8$.
\begin{theorem}[= Theorem~\ref{thm:sha-general}]\label{thm:main-general}
Let $n$ be an odd positive square-free integer coprime with $e_1e_2e_3abc$, whose prime factors are quadratic residues modulo each odd prime factor of $e_1e_2e_3abc$.
Assume that
\begin{itemize}
\item both $E$ and $E^{(n)}$ have no rational point of order $4$;
\item every prime factor of $n$ is congruent to $1$ modulo $8$.
\end{itemize}
If $\Sel_2(E/\BQ)\cong\Sel_2(\CE/\BQ)\cong(\BZ/2\BZ)^2$, then the following are equivalent:
\begin{enumerate}
\item $\rank_\BZ E^{(n)}(\BQ)=0$ and $\Sha\bigl(E^{(n)}/\BQ\bigr)[2^\infty]\cong(\BZ/2\BZ)^{2t}$;
\item $\rank_\BZ \CE^{(n)}(\BQ)=0$ and $\Sha\bigl(\CE^{(n)}/\BQ\bigr)[2^\infty]\cong(\BZ/2\BZ)^{2t}$.
\end{enumerate}
\end{theorem}

In each case, we will study the local solvability of homogeneous spaces and show the consistency of $2$-Selmer groups.
Then we will use Lemmas~\ref{lem:special-residue} and \ref{lem:fund-equality} to show the consistency of Cassels pairings.
The main difference of these theorems is the local solvability and the Cassels pairing at the place $2$.
We will also give applications for the congruent elliptic curve, see Theorems~\ref{thm:cong-odd-2} and \ref{thm:cong-even-2}.

\vspace{0.5pc}
The symbols we will use are listed here.
\begin{itemize}
\item $v_p$ the normalized $p$-adic valuation.
\item $\gcd(m_1,\dots,m_t)$ the greatest common divisor of integers $m_1,\dots,m_t$.
\item $(\alpha,\beta)_v\in\set{\pm1}$ the Hilbert symbol, $\alpha,\beta\in\BQ_v^\times$.
\item $[\alpha,\beta]_v\in\BF_2$ the additive Hilbert symbol, i.e., 
$(\alpha,\beta)_v=(-1)^{[\alpha,\beta]_v}$.
\item $\left(\frac\alpha\beta\right)=\prod_{p\mid\beta}(\alpha,\beta)_p\in\set{\pm1}$ the Jacobi symbol, where $\alpha$ is coprime with $\beta>0$.
\item $\aleg\alpha\beta=\sum_{p\mid\beta}[\alpha,\beta]_p\in\BF_2$ the additive Jacobi symbol, where $\alpha$ is coprime with $\beta>0$.
\item $m^*=(-1,m)_2 m\equiv1\bmod4$ for nonzero odd integer $m$.

\item $\Lambda=(d_1,d_2,d_3)$ a triple of square-free integers, where $d_1d_2d_3$ is a square.
\item $D_\Lambda$ the homogeneous space associated to $E$ and $\Lambda$, see~\eqref{eq:D_Lambda}.
\item $\Sel_2'(\mse)$ the pure $2$-Selmer group of $\mse$, see \eqref{eq:pure-2selmer}.
We will simply write $\Lambda\in\Sel_2'(\mse)$ the class of $\Lambda\in\Sel_2(\mse/\BQ)$ for convention.

\item ${\bf0}=(0,\dots,0)^\rmT$ and ${\bf1}=(1,\dots,1)^\rmT$.
\item $\bfI$ the identity matrix and $\bfO$ the zero matrix.
\item $\bfA=\bfA_n$ a matrix associated to $n$, see~\eqref{eq:A-n}.
\item $\bfD_u=\diag\Bigl\{\aleg u{p_1},\dots,\aleg u{p_k}\Bigr\}$, see~\eqref{eq:D-u}.
\end{itemize}

\section{The general case}

\subsection{Classical $2$-descent}

As shown in \cite{Cassels1998}, the $2$-Selmer group $\Sel_2(E/\BQ)$ can be identified with
\[\set{\Lambda=(d_1,d_2,d_3)\in \bigl(\BQ^\times/\BQ^{\times2}\bigr)^3:
D_\Lambda(\BA_\BQ)\neq\emptyset,d_1d_2d_3\equiv1\bmod\BQ^{\times2}},\]
where $D_\Lambda$ is a genus one curve defined by
\begin{equation}\label{eq:D_Lambda}
\begin{cases}
	H_1:& e_1t^2+d_2u_2^2-d_3u_3^2=0, \\
	H_2:& e_2t^2+d_3u_3^2-d_1u_1^2=0, \\
	H_3:& e_3t^2+d_1u_1^2-d_2u_2^2=0.
\end{cases}
\end{equation}
Under this identification,  the points $O,(e_1,0),(-e_2,0),(0,0)$ and other point $(x,y)\in E(\BQ)$ correspond to 
\[(1,1,1),\ (-e_3,-e_1e_3,e_1),\ (-e_2e_3,e_3,-e_2),\ (e_2,-e_1,-e_1e_2)\]
and $(x+e_2,x-e_1,x)$ respectively.

Denote by 
\begin{equation}\label{eq:pure-2selmer}
\Sel_2'(E):=\frac{\Sel_2(E/\BQ)}{E(\BQ)_\tors/2E(\BQ)_\tors}
\end{equation}
the pure $2$-Selmer group of $E$ defined over $\BQ$.

\begin{lemma}[{\cite{Ono1996}}]\label{lem:ono}
$E(\BQ)$ has a point of order $4$ if and only if one of the three pairs $(-e_1,e_2), (-e_2,e_3)$ and $(-e_3,e_1)$ consists of squares of integers.
\end{lemma}
If $E$ has no rational point of order $4$, then $\Sel_2(E/\BQ)$ contains $E(\BQ)[2^\infty]=E(\BQ)[2]\cong(\BZ/2\BZ)^2$ and therefore $\Sel_2'(E)=\Sel_2(E/\BQ)/E(\BQ)[2]$.
Cassels in \cite{Cassels1998} defined a skew-symmetric bilinear pairing $\pair{-,-}$ on the $\BF_2$-vector space $\Sel_2'(E)$.
We will write it additively.
For any $\Lambda\in\Sel_2(E)$, choose $P=(P_v)_v\in D_\Lambda(\BA_\BQ)$.
Since $H_i$ is locally solvable everywhere, there exists $Q_i\in H_i(\BQ)$ by Hasse-Minkowski principle.
Let $L_i$ be a linear form in three variables such that $L_i=0$ defines the tangent plane of $H_i$ at $Q_i$.
For any $\Lambda'=(d_1',d_2',d_3')\in \Sel_2(E)$, define
\[\pair{\Lambda,\Lambda'}_E=\sum_v \pair{\Lambda,\Lambda'}_{E,v}\in\BF_2,\quad\text{where}\quad \pair{\Lambda,\Lambda'}_{E,v}=\sum_{i=1}^3 \bigl[L_i(P_p), d_i'\bigr]_v.\]
This pairing is independent of the choice of $P$ and $Q_i$, and is trivial on $E(\BQ)[2]$.
We will omit the subscript $E$ if there is no confusion.

\begin{lemma}[{\cite[Lemma 7.2]{Cassels1998}}]\label{lem:cassels}
The local Cassels pairing $\pair{\Lambda,\Lambda'}_{E,p}=0$ if
\begin{itemize}
\item $p\nmid 2\infty$,
\item the coefficients of $H_i$ and $L_i$ are all integral at $p$, and
\item modulo $D_\Lambda$ and $L_i$ by $p$, they define a curve of genus $1$ over $\BF_p$ together with tangents to it.
\end{itemize}
\end{lemma}

\begin{lemma}[{\cite[p.~2157]{Wang2016}}]\label{lem:non-deg}
If $E$ has no rational point of order $4$, then the following are equivalent:
\begin{enumerate}
\item $\rank_\BZ E(\BQ)=0$ and $\Sha(E/\BQ)[2^\infty]\cong (\BZ/2\BZ)^{2t}$;
\item $\Sel_2'(E)$ has dimension $2t$ and the Cassels pairing on it is non-degenerate.
\end{enumerate}
\end{lemma}

\subsection{Homogeneous spaces}
Let's consider the quadratic twist $E^{(n)}$.
The homogeneous space $D_\Lambda^{(n)}$ associated to $\Lambda=(d_1,d_2,d_3)$ is
\[\begin{cases}
	H_1:& e_1nt^2+d_2u_2^2-d_3u_3^2=0, \\
	H_2:& e_2nt^2+d_3u_3^2-d_1u_1^2=0, \\
	H_3:& e_3nt^2+d_1u_1^2-d_2u_2^2=0.
\end{cases}\]
By classical descent theory, if $p\nmid 2e_1e_2e_3n$, then $D_\Lambda^{(n)}(\BQ_p)$ is non-empty if and only if $p\nmid d_1d_2d_3$, see \cite[Theorem~X.1.1, Corollary~X.4.4]{Silverman2009}.
Hence we may assume that $d_1,d_2,d_3$ are square-free divisors of $2e_1e_2e_3n$ from now on.

\begin{lemma}\label{lem:local-solv-R}
Let $\Lambda=(d_1,d_2,d_3)$.
Then $D_\Lambda^{(n)}(\BR)\neq\emptyset$ if and only if
\begin{itemize}
\item $d_1>0$, if $e_2>0,e_3<0$;
\item $d_2>0$, if $e_3>0,e_1<0$;
\item $d_3>0$, if $e_1>0,e_2<0$.
\end{itemize}
\end{lemma}
\begin{proof}
The proof is similar to \cite[Lemma~3.1(4)]{WangZhang2022}, which is easy to get.
\end{proof}

\begin{lemma}\label{lem:local-solv-p}
Let $\Lambda=(d_1,d_2,d_3)$ with square-free $d_i$.
Let $n$ be a positive square-free integer coprime with $e_1e_2e_3$ and $p$ an odd prime factor of $n$.
Then $D_\Lambda^{(n)}(\BQ_p)\neq\emptyset$ if and only if
\begin{itemize}
\item $\leg{d_1}p=\leg{d_2}p=\leg{d_3}p=1$, if $p\nmid d_1d_2d_3$;
\item $\leg{-e_2e_3d_1}p=\leg{e_3n/d_2}p=\leg{-e_2 n/d_3}p=1$, if $p\nmid d_1,p\mid d_2,p\mid d_3$;
\item $\leg{-e_3n/d_1}p=\leg{-e_3e_1d_2}p=\leg{e_1 n/d_3}p=1$, if $p\mid d_1,p\nmid d_2,p\mid d_3$;
\item $\leg{e_2n/d_1}p=\leg{-e_1n/d_2}p=\leg{-e_1e_2d_3}p=1$, if $p\mid d_1,p\mid d_2,p\nmid d_3$.
\end{itemize}
\end{lemma}
\begin{proof}
Assume that $p\nmid d_1, p\nmid d_2, p\nmid d_3$.
If $D_\Lambda^{(n)}(\BQ_p)\neq\emptyset$, then each $H_i(\BQ_p)\neq\emptyset$ and $\leg{d_2d_3}p=\leg{d_1d_3}p=\leg{d_1d_2}p=1$.
That's to say, $\leg{d_1}p=\leg{d_2}p=\leg{d_3}p=1$.
Conversely, if $\leg{d_1}p=\leg{d_2}p=\leg{d_3}p=1$, then $(0,\sqrt{1/d_1},\sqrt{1/d_2},\sqrt{1/d_3})\in D_\Lambda^{(n)}(\BQ_p)$.

Assume that $p\nmid d_1, p\mid d_2, p\mid d_3$.
Then $D_\Lambda^{(n)}(\BQ_p)\neq\emptyset$ if and only if $D_{\Lambda'}^{(n)}(\BQ_p)\neq\emptyset$, where 
Hence this case can be reduced to the case $p\nmid d_1d_2d_3$.
The rest cases can be obtained by symmetry.
\end{proof}

Let $n=p_1\cdots p_k$ be a prime decomposition of $n$.
For $\Lambda=(d_1,d_2,d_3)$ with square-free $d_i\mid 2e_1e_2e_3n$, denote by
\begin{equation}\label{eq:xyz}
x_i=v_{p_i}(d_1),\quad y_i=v_{p_i}(d_2), \quad z_i=v_{p_i}(d_3).\end{equation}
Then $\bfx+\bfy+\bfz={\bf0}$, where
\[\bfx=(x_1,\dots,x_k)^\rmT,\quad
\bfy=(y_1,\dots,y_k)^\rmT,\quad
\bfz=(z_1,\dots,z_k)^\rmT\in\BF_2^k.\]
Write
\begin{equation}\label{eq:case-general}
\begin{split}
d_1&=p_1^{x_1}\cdots p_k^{x_k}\cdot \wt d_1,\\
d_2&=p_1^{y_1}\cdots p_k^{y_k}\cdot \wt d_2,\\
d_3&=p_1^{z_1}\cdots p_k^{z_k}\cdot \wt d_3.
\end{split}
\end{equation}
Then $\wt d_1\wt d_2\wt d_3\in\BQ^{\times2}$.

Denote by 
\begin{equation}\label{eq:A-n}
\bfA=\bfA_n=\bigl([p_j,-n]_{p_i}\bigr)_{i,j}\in M_k(\BF_2)
\end{equation}
and
\begin{equation}\label{eq:D-u}
\bfD_u=\diag\set{\aleg u{p_1},\cdots,\aleg u{p_k}}\in M_k(\BF_2).
\end{equation}

\begin{theorem}\label{thm:selmer-general}
Let $n$ be an odd positive square-free integer coprime with $e_1e_2e_3$, whose prime factors are quadratic residues modulo each odd prime factor of $e_1e_2e_3$.
Assume that 
\begin{itemize}
\item both $E$ and $E^{(n)}$ have no rational point of order $4$;
\item every prime factor of $n$ is congruent to $1$ modulo $8$.
\end{itemize}
If $\Sel_2(E/\BQ)\cong(\BZ/2\BZ)^2$, then the map $(d_1,d_2,d_3)\mapsto \svec{\bfx}{\bfy}$ induces an isomorphism
\[\Sel_2'\bigl(E^{(n)}\bigr)\simto \Ker\begin{pmatrix}
\bfA&\\&\bfA
\end{pmatrix},\]
where $0<d_i\mid n$.
\end{theorem}

\begin{proof}
Let $\Lambda=(d_1,d_2,d_3)$ with square-free $d_i\mid 2e_1e_2e_3n$ and denote by $\wt\Lambda=(\wt d_1,\wt d_2,\wt d_3)$.
Then $D_\Lambda^{(n)}(\BR)\neq\emptyset$ if and only if $D_{\wt\Lambda}^{(1)}(\BR)\neq\emptyset$ by Lemma~\ref{lem:local-solv-R} and the fact $\sgn(\wt d_i)=\sgn(d_i)$.

If $q$ is a prime factor of $2e_1e_2e_3$, then $n, d_i/\wt d_i\in\BQ_q^{\times2}$.
Therefore, 
\[(t,u_1,u_2,u_3)\in D_\Lambda^{(n)}(\BQ_q)\iff\left(t\sqrt{n},u_1\sqrt{\frac{d_1}{\wt d_1}},u_2\sqrt{\frac{d_2}{\wt d_2}},u_3\sqrt{\frac{d_3}{\wt d_3}}\right)\in D_{\wt \Lambda}^{(1)}(\BQ_q).\]
Hence $\Lambda\in\Sel_2\bigl(E^{(n)}/\BQ\bigr)$ if and only if $\wt \Lambda\in\Sel_2(E/\BQ)$ and $D_\Lambda^{(n)}$ is locally solvable at each $p\mid n$.

If $\Lambda\in\Sel_2\bigl(E^{(n)}/\BQ\bigr)$, then $\wt\Lambda\in\Sel_2(E/\BQ)$.
By our assumptions, 
\[\wt\Lambda=(1,1,1),\ (-e_3,-e_1e_3,e_1),\ (-e_2e_3,e_3,-e_2)\ \text{or}\ (e_2,-e_1,-e_1e_2)\]
is $2$-torsion.
If $\wt\Lambda=(-e_3,-e_1e_3,e_1)$, then
\[\Lambda\cdot(-e_3n,-e_1e_3,e_1n)=\Bigl(\prod_{i=1}^k p_i^{1-x_i},\prod_{i=1}^k p_i^{y_i},\prod_{i=1}^k p_i^{1-z_i}\Bigr).\]
The other cases are similar.
Hence each element in $\Sel_2'\bigl(E^{(n)}\bigr)$ has a unique representative $(d_1,d_2,d_3)$ with $0<d_i\mid n$.
Based on this, we can express $\Sel_2'\bigl(E^{(n)}\bigr)$ in terms of linear algebra by Lemma~\ref{lem:local-solv-p} after a translation of languages:
\[\Sel_2'\bigl(E^{(n)}\bigr)\simto \bfM_n,
\quad\text{where}\quad
\bfM_n=\begin{pmatrix}
\bfA+\bfD_{-e_3}&\bfD_{-e_2e_3}\\
\bfD_{-e_1e_3}&\bfA+\bfD_{e_3}
\end{pmatrix}.\]
Since $\leg pq=1$ for any odd primes $p\mid n,q\mid e_1e_2e_3$ and $\leg{\pm1}p=\leg{\pm2}p=1$, we have $\leg{\pm e_i}{p}=1$.
Therefore, $\bfD_{\pm e_i}=\bfO$ and $\bfM_n=\diag\{\bfA,\bfA\}$.
\end{proof}

\subsection{The Cassels pairing}

Let $(a,b,c)$ be a primitive triple of odd integers satisfying
\[e_1a^2+e_2b^2+e_3c^2=0.\]
Denote by $\CE=\mse_{e_1a^2,e_2b^2}$ and $\CE^{(n)}=\mse_{e_1a^2n,e_2b^2n}$.

\begin{theorem}\label{thm:sha-general}
Let $n$ be an odd positive square-free integer coprime with $e_1e_2e_3abc$, whose prime factors are quadratic residues modulo each odd prime factor of $e_1e_2e_3abc$.
Assume that
\begin{itemize}
\item both $E$ and $E^{(n)}$ have no rational point of order $4$;
\item every prime factor of $n$ is congruent to $1$ modulo $8$.
\end{itemize}
If $\Sel_2(E/\BQ)\cong\Sel_2(\CE/\BQ)\cong(\BZ/2\BZ)^2$, then the following are equivalent:
\begin{enumerate}
\item $\rank_\BZ E^{(n)}(\BQ)=0$ and $\Sha\bigl(E^{(n)}/\BQ\bigr)[2^\infty]\cong(\BZ/2\BZ)^{2t}$;
\item $\rank_\BZ \CE^{(n)}(\BQ)=0$ and $\Sha\bigl(\CE^{(n)}/\BQ\bigr)[2^\infty]\cong(\BZ/2\BZ)^{2t}$.
\end{enumerate}
\end{theorem}

\begin{lemma}\label{lem:special-residue}
Assume that all prime factors of $n$ are nonzero quadratic residues modulo each odd prime factor of $e_1e_2e_3$.
If $a\equiv b\equiv c\equiv 1\bmod 4$, then 
\[\frac18(a+b)(b+c)(c+a)\equiv1\bmod 4\]
is a quadratic residue modulo each prime factor of $n$.
\end{lemma}
\begin{proof}
Let $\alpha,\beta$ be coprime integers satisfying
\[\frac\beta\alpha=\frac{e_1(a-c)}{e_2(b+c)}.\]
Then $\alpha$ is odd and $\beta$ is even.
It's not hard to show that
\[\begin{split}
\lambda a&=e_1\alpha^2+2e_2\alpha\beta-e_2\beta^2\equiv e_1\bmod 4,\\
\lambda b&=e_1\alpha^2-2e_1\alpha\beta-e_2\beta^2\equiv e_1\bmod 4,\\
\lambda c&=e_1\alpha^2+e_2\beta^2\equiv e_1\bmod 4,
\end{split}\]
for some $\lambda\equiv e_1\bmod 4$.
Then
\[\begin{split}
\lambda(a+b)&=2(\alpha-\beta)(e_1\alpha+e_2\beta),\\
\lambda(b+c)&=2e_1\alpha(\alpha-\beta),\\
\lambda(c+a)&=2\alpha(e_1\alpha+e_2\beta)
\end{split}\]
and
\[\frac18(a+b)(b+c)(c+a)=e_1\lambda\bigl(\lambda^{-2}\alpha(\alpha-\beta)(e_1\alpha+e_2\beta)\bigr)^2\equiv1\bmod4.\]

Let $q$ be a prime factor of $\lambda$.
Then 
\[q\mid \gcd\bigl(\lambda(a+b),\lambda(a+c)\bigr)=2(e_1\alpha+e_2\beta).\]
If $q\nmid e_1$, then $q\mid \alpha(\alpha-\beta)$.
If $q\mid \alpha$, then $q\mid e_2\beta$, $q\mid e_2$;
if $q\mid(\alpha-\beta)$, then $q\mid e_2(\alpha-\beta)+(e_1\alpha+e_2\beta)=-e_3\alpha$, $q\mid e_3$.
Hence $q\mid e_1e_2e_3$.

Let $p$ be a prime factor of $n$.
Since $e_1\lambda\equiv1\bmod 4$ and $\leg{p}{q}=1$ for any odd prime $q\mid e_1e_2e_3$, we have
\[\leg{e_1\lambda}{p}=\leg{p}{e_1\lambda}=\prod_{q\mid e_1\lambda}\leg pq^{v_q(e_1\lambda)}=1.\]
Hence $(a+b)(b+c)(c+a)/8$ is a quadratic residue modulo $p$.
\end{proof}

\begin{lemma}\label{lem:fund-equality}
We have
\[\begin{split}
&(ax+by+cz)(x+y+z)-\half(a+b)(b+c)(c+a)\biggl(
\frac x{b+c}+
\frac y{c+a}+
\frac z{a+b}\biggr)^2\\
&\quad=\half(e_1a+e_2b+e_3c)\biggl(\frac{x^2}{e_1}+\frac{y^2}{e_2}+\frac{z^2}{e_3}\biggr).
\end{split}\]
\end{lemma}
\begin{proof}
The coefficient of $x^2$ on the left hand side is
\[\begin{split}
&a-\frac{(a+b)(a+c)}{2(b+c)}
=\frac{a(b+c)-bc-a^2}{2(b+c)}
=\frac{e_1a(b+c)-e_1bc-e_1a^2}{2e_1(b+c)}\\
=&\frac{e_1a(b+c)+(e_2+e_3)bc+e_2b^2+e_3c^2}{2e_1(b+c)}
=\frac{e_1a+e_2b+e_3c}{2e_1}
\end{split}\]
and the coefficient of $yz$ on the left hand side is zero.
The equality then follows by symmetry.
\end{proof}

\begin{proof}[Proof of Theorem~\ref{thm:sha-general}]
Since $E$ has no rational point of order $4$, none of $(-e_1,e_2)$, $(-e_2,e_3)$, $(-e_3,e_1)$ consists of squares by Lemma~\ref{lem:ono}.
Therefore, none of $(-e_1a^2,e_2b^2)$, $(-e_2b^2,e_3c^2)$, $(-e_3c^2,e_1a^2)$ consists of squares and $\CE$ has no rational point of order $4$.
Similarly, $\CE^{(n)}$ has no rational point of order $4$.

By choosing suitable signs, we may assume that $a\equiv b \equiv c\equiv 1\bmod 4$.
Since the matrix in Theorem~\ref{thm:selmer-general} does not depend on $a,b,c$, we have a canonical isomorphism
\[\Sel_2'\bigl(E^{(n)}\bigr)\cong\Sel_2'\bigl(\CE^{(n)}\bigr).\]
Let $\Lambda=(d_1,d_2,d_3),\Lambda'=(d_1',d_2',d_3')\in\Sel_2'\bigl(E^{(n)}\bigr)$ with $0<d_i,d_i'\mid n$.
We will denote by $D,H,Q,L,P$ the corresponding symbols for $E$ and $\CD,\CH,\CQ,\CL,\CP$ the corresponding symbols for $\CE$ in the calculation of Cassels pairing.
Then $\CD_\Lambda^{(n)}$ is defined as
\[\begin{cases}
	\CH_1:& e_1a^2nt^2+d_2u_2^2-d_3u_3^2=0, \\
	\CH_2:& e_2b^2nt^2+d_3u_3^2-d_1u_1^2=0, \\
	\CH_3:& e_3c^2nt^2+d_1u_1^2-d_2u_2^2=0.
\end{cases}\]
Let $(\alpha_i,\beta_i,\gamma_i)$ be primitive triples of integers satisfying
\[\begin{split}
e_1n\alpha_1^2+d_2\beta_1^2-d_3\gamma_1^2&=0, \\
e_2n\alpha_2^2+d_3\beta_2^2-d_1\gamma_2^2&=0, \\
e_3n\alpha_3^2+d_1\beta_3^2-d_2\gamma_3^2&=0.
\end{split}\]
Choose
\[\begin{aligned}
\CQ_1&=(\alpha_1,a\beta_1,a\gamma_1)\in \CH_1(\BQ),&
\CL_1&=e_1an\alpha_1t+d_2\beta_1u_2-d_3\gamma_1u_3,\\
\CQ_2&=(\alpha_2,b\beta_2,b\gamma_2)\in \CH_2(\BQ),&
\CL_2&=e_2bn\alpha_2t+d_3\beta_2u_3-d_1\gamma_2u_1,\\
\CQ_3&=(\alpha_3,c\beta_3,c\gamma_3)\in \CH_3(\BQ),&
\CL_3&=e_3cn\alpha_3t+d_1\beta_3u_1-d_2\gamma_3u_2.
\end{aligned}\]

(i) The case $v\mid 2e_1e_2e_3abc$.
Since each prime factor of $n$ is a square in $\BQ_v$, so is $d_i'$. Therefore, $[\CL_i(\CP_v),d_i']_v=0=[L_i(P_v),d_i']_v$.

(ii) The case $v=p\mid n$.
Since $a\equiv 1\bmod 4$ and $p$ is a quadratic residue modulo every odd prime factor $q$ of $abc$, we have 
\[[a,p]_p=\Bigl[\frac ap\Bigr]=\Bigl[\frac pa\Bigr]=\sum_{q\mid a}v_q(a)\Bigl[\frac pq\Bigr]=0.\]
Therefore $[a,d_i']_p=0$.
Similarly, $[b,d_i']_p=[c,d_i']_p=0$.

(ii-a) The case $p\nmid d_1d_2d_3$.
Take $\CP_p=(0,1/\sqrt{d_1},1/\sqrt{d_2},1/\sqrt{d_3})=P_p$.
Then
\[\CL_1(\CP_p)=\beta_1\sqrt{d_2}-\gamma_1\sqrt{d_3}=L_1(P_p).\]
Similarly, $\CL_2(\CP_p)=L_2(P_p)$ and $\CL_3(\CP_p)=L_3(P_p)$.

(ii-b) The case $p\nmid d_1, p\mid d_2, p\mid d_3$.
Then $e_3n/d_2,-e_2n/d_3\in\BQ_p^{\times2}$ by Lemma~\ref{lem:local-solv-p}.
Take $\CP_p=(1,0,cu,bv)$ where $u^2=e_3n/d_2,v^2=-e_2n/d_3$.
Then $P_p=(1,0,u,v)$ and
\[\begin{split}
\CL_1(\CP_p)&=ae_1n\alpha_1-bd_3\gamma_1v+cd_2\beta_1u,\\
\CL_2(\CP_p)&=be_2n\alpha_2+bd_3\beta_2v=bL_2(P_p),\\
\CL_3(\CP_p)&=ce_3n\alpha_3-cd_2\gamma_3u=cL_3(P_p).
\end{split}\]
Since
\[\frac{(e_1n\alpha_1)^2}{e_1}+
\frac{(-d_3\gamma_1v)^2}{e_2}+
\frac{(d_2\beta_1u)^2}{e_3}
=n(e_1n\alpha_1^2-d_3\gamma_1^2+d_2\beta_1^2)=0,\]
we have
\[\CL_1(\CP_p)L_1(P_p)=\half(a+b)(a+c)(b+c)\biggl(
\frac{e_1n\alpha_1}{b+c}+
\frac{d_2\beta_1u}{a+b}-
\frac{d_3\gamma_1v}{a+c}\biggr)^2\]
by Lemma~\ref{lem:fund-equality}.
Therefore, 
\[\begin{split}
[\CL_1(\CP_p),d_1']_p&=[L_1(P_p),d_1']_p+[2(a+b)(a+c)(b+c),d_1']_p=[L_1(P_p),d_1']_,\\
[\CL_2(\CP_p),d_2']_p&=[L_2(P_p),d_2']_p+[b,d_2']_p=[L_2(P_p),d_2']_p,\\
[\CL_3(\CP_p),d_3']_p&=[L_3(P_p),d_3']_p+[c,d_3']_p=[L_3(P_p),d_3']_p
\end{split}\]
by Lemma~\ref{lem:special-residue}.

(ii-c) The case $p\mid d_1, p\nmid d_2, p\mid d_3$, and the case $p\mid d_1, p\mid d_2, p\nmid d_3$ can be proved similarly by the symmetry of $e_i$.

\vspace{0.5em}
Now we have
\[\begin{split}
\pair{\Lambda,\Lambda'}_{\CE^{(n)}}
&=\sum_{v\mid 2e_1e_2e_3abcn\infty} \sum_{i=1}^3\bigl[\CL_i(\CP_v),d_i'\bigr]_v
=\sum_{p\mid n} \sum_{i=1}^3\bigl[\CL_i(\CP_p),d_i'\bigr]_p\\
&=\sum_{p\mid n} \sum_{i=1}^3\bigl[L_i(P_p),d_i'\bigr]_p
=\pair{\Lambda,\Lambda'}_{E^{(n)}}
\end{split}\]
by Lemma~\ref{lem:cassels}.
In other words, the Cassels pairings on $\Sel_2'\bigl(E^{(n)}\bigr)$ and $\Sel_2'\bigl(\CE^{(n)}\bigr)$ are same under the identity $\Sel_2'\bigl(E^{(n)}\bigr)\cong\Sel_2'\bigl(\CE^{(n)}\bigr)$.
Since both $E^(n)$ and $\CE^{(n)}$ have no rational point of order $4$, this theorem follows from Lemma~\ref{lem:non-deg}.
\end{proof}

\section{The odd case with \texorpdfstring{$2\mmid e_3$}{2||e_3}}

Assume that $e_1,e_2$ are odd and $2\mmid e_3$.
Let $n$ be an odd positive square-free integer.
Let $\Lambda=(d_1,d_2,d_3)$ where $d_1,d_2,d_3$ are square-free integers dividing $2e_1e_2e_3n$.

\subsection{Homogeneous spaces}

\begin{lemma}\label{lem:local-solv-odd}
If $D_\Lambda^{(n)}(\BQ_2)\neq\emptyset$, then $d_3$ is odd.
\end{lemma}
\begin{proof}
The proof is similar to \cite[Lemma~3.1(2)]{WangZhang2022}.
Since we are dealing with homogeneous spaces, we may assume that $t,u_1,u_2,u_3$ are $2$-adic integers and at least one of them is a $2$-adic unit.
Suppose that $D_\Lambda^{(n)}(\BQ_2)\neq\emptyset$.
If $2\mid d_1, 2\nmid d_2, 2\mid d_3$, then $u_2$ is even by $H_3$ and $t$ is even $H_2$.
Therefore, $u_3$ is even by $H_1$ and $u_1$ is even by $H_2$, which is impossible.
The case $2\nmid d_1, 2\mid d_2, 2\mid d_3$ is similar.
Hence $d_3$ is odd.
\end{proof}

Since the torsion $(-e_3n,-e_1e_3,e_1n)$ has $2$-adic valuation $(1,1,0)$, any element in the pure $2$-Selmer group  $\Sel_2'\bigl(E^{(n)}\bigr)$ has a representative $\Lambda=(d_1,d_2,d_3)$ with odd $d_i\mid e_1e_2e_3n$.

\begin{lemma}\label{lem:local-solv2-odd}
Let $\Lambda=(d_1,d_2,d_3)$ where $d_1,d_2,d_3$ are odd.
If $D_\Lambda^{(n)}$ is locally solvable at all places $v\neq 2$, then $D_\Lambda^{(n)}$ is also locally solvable at $v=2$.
\end{lemma}
\begin{proof}
The proof is similar to \cite[Lemma~3.4]{WangZhang2022}.
Since $D_\Lambda^{(n)}(\BQ_v)\neq\emptyset$ for all places $v\neq 2$, each $H_i$ is locally solvable at $v\neq 2$.
By the product formula of Hilbert symbols, $H_i$ is also locally solvable at $2$.
In other words, 
\[[e_1nd_3,d_1]_2=[e_2nd_1,d_2]_2=[e_3nd_2,d_3]_2=0.\]

(i) If $(d_1,d_2,d_3)\equiv(1,1,1)\bmod4$, then $0=[e_3nd_2,d_3]_2=[2,d_3]_2$ and we have $d_3\equiv1\bmod 8$.
Therefore, $d_1\equiv d_2\bmod 8$.
If $d_1\equiv d_2\equiv1\bmod 8$, take
\[t=0,\ u_1=\sqrt{d_3/d_1},\ u_2=\sqrt{d_3/d_2},\ u_3=1.\]
If $d_1\equiv d_2\equiv5\bmod 8$, take
\[t=2,\ u_1=\sqrt{(d_3+4e_2n)/d_1},\ u_2=\sqrt{(d_3-4e_1n)/d_2},\ u_3=1.\]

(ii) If $(d_1,d_2,d_3)\equiv(-1,-1,1)\bmod4$, then $d_3\equiv 1\bmod 8$ similarly.
Since 
\[[e_1n,-1]_2=[e_1nd_3,d_1]_2=0=[e_2nd_1,d_2]_2=[-e_2n,-1]_2,\]
we have $e_1n\equiv -e_2n\equiv 1\bmod 4$.
This implies that $4\mid(e_1+e_2)=-e_3$, which is impossible.

(iii) If $d_3\equiv-1\bmod4$, then $[e_3nd_2,d_3]_2=0$, $e_3nd_2\equiv d_3+3\bmod8$ and
\[(d_1-e_2n)-(d_2+e_1n)=d_1-d_2+e_3n\equiv 2(d_1+d_2)\equiv 0\bmod 8.\]
If $(d_1,d_2,d_3)\equiv(1,-1,-1)\bmod4$, then $[e_2n,-1]_2=0$ and $e_2n\equiv d_1\bmod 4$.
If $(d_1,d_2,d_3)\equiv(-1,1,-1)\bmod4$, then $[-e_1n,-1]_2=0$ and $e_1n\equiv-d_2\bmod 4$.
If $d_2+e_1n\equiv d_1-e_2n\equiv 0\bmod 8$, take
\[t=1,\ u_1=\sqrt{e_2n/d_1},\ u_2=\sqrt{-e_1n/d_2},\ u_3=0.\]
If $d_2+e_1n\equiv d_1-e_2n\equiv 4\bmod 8$, take
\[t=1,\ u_1=\sqrt{(4d_3+e_2n)/d_1},\ u_2=\sqrt{(4d_3-e_1n)/d_2},\ u_3=2.\]

Hence $D_\Lambda^{(n)}$ is locally solvable at $v=2$.
\end{proof}

Let $\Lambda=(d_1,d_2,d_3)$ with odd square-free $d_i\mid e_1e_2e_3n$.
We will use the notations $\bfx,\bfy,\bfz,\wt d_i$ in \eqref{eq:xyz} and \eqref{eq:case-general}.

\begin{theorem}\label{thm:selmer-odd}
Let $n$ be an odd positive square-free integer coprime with $e_1e_2e_3$, whose prime factors are quadratic residues modulo each odd prime factor of $e_1e_2e_3$.
If $\Sel_2(E/\BQ)\cong(\BZ/2\BZ)^2$, then the map $(d_1,d_2,d_3)\mapsto \svec{\bfx}{\bfy}$ induces an isomorphism
\[\Sel_2'\bigl(E^{(n)}\bigr)\simto \Ker \begin{pmatrix}
\bfA+\bfD_{-e_3}&\bfD_{-e_2e_3}\\
\bfD_{-e_1e_3}&\bfA+\bfD_{e_3}
\end{pmatrix},\]
where $0<d_i\mid n$.
\end{theorem}

\begin{proof}
Since $e_1,e_2$ are odd and $2\mmid e_3$, neither $(-ne_2,ne_3)$ nor $(-ne_3,ne_1)$ consists of squares.
If $(-ne_1,ne_2)$ consists of squares, then $e_1\equiv -e_2\bmod 4$ and $4\mid e_3$, which is impossible.
Hence $E(\BQ)$ contains no point of order $4$ by Lemma~\ref{lem:ono}.

Let $\Lambda=(d_1,d_2,d_3)$ with odd square-free $d_i\mid e_1e_2e_3n$ and denote by $\wt \Lambda=(\wt d_1,\wt d_2,\wt d_3)$.
Similar to the proof of Theorem~\ref{thm:selmer-general},
$D_\Lambda^{(n)}(\BQ_v)\neq\emptyset$ if and only if $D_{\wt\Lambda}^{(1)}(\BQ_v)\neq\emptyset$ for $v=\infty$ or odd $v\mid e_1e_2e_3$.
Hence $\Lambda\in\Sel_2\bigl(E^{(n)}/\BQ\bigr)$ if and only if $\wt \Lambda\in\Sel_2(E/\BQ)$ and $D_\Lambda^{(n)}$ is locally solvable at each $p\mid n$ by Lemmas~\ref{lem:local-solv-odd} and \ref{lem:local-solv2-odd}.

If $\Lambda\in\Sel_2\bigl(E^{(n)}/\BQ\bigr)$, then $\wt\Lambda\in\Sel_2(E/\BQ)$.
By our assumptions, $\wt\Lambda$ is $2$-torsion, which should be $(1,1,1)$ or $(e_2,-e_1,-e_1e_2)$.
If $\wt\Lambda=(e_2,-e_1,-e_1e_2)$, then
\[\Lambda\cdot(ne_2,-ne_1,-e_1e_2)=\Bigl(\prod_{i=1}^k p_i^{1-x_i},\prod_{i=1}^k p_i^{1-y_i},\prod_{i=1}^k p_i^{z_i}\Bigr).\]
Hence each element in $\Sel_2'\bigl(E^{(n)}\bigr)$ has a unique representative $(d_1,d_2,d_3)$ with $0<d_i\mid n$.
Based on this, we can express $\Sel_2'\bigl(E^{(n)}\bigr)$ in terms of linear algebra by Lemma~\ref{lem:local-solv-p} after a translation of languages.
\end{proof}
\begin{remark}
Since $\leg pq=1$ for any odd primes $p\mid n, q\mid e_1e_2e_3$, we have $\bfD_e=\bfD_u$, where $u\in\set{\pm1,\pm2}$ such that $e/u\equiv 1\bmod 4$ for any square-free $e\mid e_1e_2e_3$.
\end{remark}

\subsection{The Cassels pairing}
Let $(a,b,c)$ be a primitive triple of integers satisfying
\[e_1a^2+e_2b^2+e_3c^2=0.\]
Then $a,b,c$ are odd.
Denote by $\CE=\mse_{e_1a^2,e_2b^2}$ and $\CE^{(n)}=\mse_{e_1a^2n,e_2b^2n}$.

\begin{theorem}\label{thm:sha-odd}
Let $n$ be an odd positive square-free integer coprime with $e_1e_2e_3abc$, whose prime factors are quadratic residues modulo each odd prime factor of $e_1e_2e_3abc$.
If $\Sel_2(E/\BQ)\cong\Sel_2(\CE/\BQ)\cong(\BZ/2\BZ)^2$, then the following are equivalent:
\begin{enumerate}
\item $\rank_\BZ E^{(n)}(\BQ)=0$ and $\Sha\bigl(E^{(n)}/\BQ\bigr)[2^\infty]\cong(\BZ/2\BZ)^{2t}$;
\item $\rank_\BZ \CE^{(n)}(\BQ)=0$ and $\Sha\bigl(\CE^{(n)}/\BQ\bigr)[2^\infty]\cong(\BZ/2\BZ)^{2t}$.
\end{enumerate}
\end{theorem}

\begin{proof}
As shown in the proof of Theorem~\ref{thm:selmer-odd}, 
both $E^{(n)}$ and $\CE(\BQ)^{(n)}$ have no rational point of order $4$.
Since the matrix in Theorem~\ref{thm:selmer-odd} does not depend on $a,b,c$, we have a canonical isomorphism
\[\Sel_2'\bigl(E^{(n)}\bigr)\cong\Sel_2'\bigl(\CE^{(n)}\bigr).\]
By choosing suitable signs, we may assume that $a\equiv b \equiv c\equiv 1\bmod 4$.
Let $\Lambda=(d_1,d_2,d_3),\Lambda'=(d_1',d_2',d_3')\in\Sel_2'\bigl(E^{(n)}\bigr)$ with $0<d_i,d_i'\mid n$.
We will use the notations $\CD,\CH,\CQ,\CL,\CP,D,H,Q,L,P,\alpha_i,\beta_i,\gamma_i$ in the proof of Theorem~\ref{thm:sha-general}.

(i) The case odd $v\mid e_1e_2e_3abc n$.
The proof is similar to the proof of Theorem~\ref{thm:sha-general}.

(ii) The case $v=2$.
As shown in Lemma~\ref{lem:local-solv2-odd}, the case $(d_1,d_2,d_3)\equiv(-1,-1,1)\bmod 4$ is impossible.

(ii-a) The case $(d_1,d_2,d_3)\equiv(1,1,1)\bmod 4$.
As shown in Lemma~\ref{lem:local-solv2-odd}, if $d_1\equiv d_2\equiv 1\bmod8$, take $\CP_2=(0,1/\sqrt{d_1},1/\sqrt{d_2},1/\sqrt{d_3})=P_2$.
Then
\[\begin{split}
\CL_1(\CP_2)&=\beta_1\sqrt{d_2}-\gamma_1\sqrt{d_3}=L_1(P_2),\\
\CL_2(\CP_2)&=\beta_2\sqrt{d_3}-\gamma_2\sqrt{d_1}=L_2(P_2),\\
\CL_3(\CP_2)&=\beta_3\sqrt{d_1}-\gamma_3\sqrt{d_2}=L_3(P_2).
\end{split}\]

If $d_1\equiv d_2\equiv 5\bmod 8$, 
denote by 
\[\begin{aligned}
\CU&=\sqrt{(d_3+4e_2b^2n)d_1}, &\CV&=\sqrt{(d_3-4e_1a^2n)d_2},\\
U&=\sqrt{(d_3+4e_2n)d_1}, &V&=\sqrt{(d_3-4e_1n)d_2} 
\end{aligned}\]
with $\CU\equiv \CV\equiv U\equiv V\equiv 1\bmod 4$.
Since $\CU^2\equiv U^2\bmod{32}$, we have $\CU\equiv U\bmod {16}$.
Similarly, $\CV\equiv V\bmod{16}$.
Take $\CP_2=(2,\CU/d_1,\CV/d_2,1)$, then $P_2=(2,U/d_1,V/d_2,1)$ and
\[\begin{split}
\CL_1(\CP_2)&\equiv2e_1an\alpha_1+\beta_1 V-d_3\gamma_1\equiv L_1(P_2),\\
\CL_2(\CP_2)&\equiv2e_2bn\alpha_2+d_3\beta_2-\gamma_2U\equiv L_2(P_2),\\
\CL_3(\CP_2)&\equiv2e_3cn\alpha_3+\beta_3U-\gamma_3V\equiv L_3(P_2)
\end{split}\]
modulo $8$.
If $\alpha_1$ is odd, then exactly one of $\beta_1$ and $\gamma_1$ is odd.
Thus $\CL_1(\CP_2)$ is odd.
If $\alpha_1$ is even, then both of $\beta_1$ and $\gamma_1$ are odd.
By choosing a suitable sign of $\gamma_1$, we may assume that $2\mmid(\beta_1-\gamma_1)$.
Therefore, $2\mmid \CL_1(\CP_2)$.
Similarly, we may assume that $2\mmid \CL_2(\CP_2)$.
Note that $\beta_3,\gamma_3$ are odd.
By choosing a suitable sign of $\gamma_3$, we may assume that $2\mmid \CL_3(\CP_2)$.
Since $\CL_i(\CP_2)\equiv L_i(P_2)\bmod8$, we have
\[[\CL_i(\CP_2),d_i']_2=[L_i(P_2),d_i']_2.\]

(ii-b) The case $d_3\equiv-1\bmod4$.
As shown in Lemma~\ref{lem:local-solv2-odd}, 
\[e_1n+d_2\equiv e_2n-d_1\equiv 0\bmod 4
\quad\text{and}\quad
(e_1n+d_2)-(e_2n-d_1)\equiv 0\bmod 8.\]
If $e_1n+d_2\equiv e_2n-d_1\equiv 0\bmod 8$, take $\CP_2=(1,bu/d_1,av/d_2,0)$ where $u^2=e_2nd_1,v^2=-e_1nd_2$.
Then $P_2=(1,u/d_1,v/d_2,0)$ and
\[\begin{split}
\CL_1(\CP_2)&=ae_1n\alpha_1+a\beta_1v=aL_1(P_2),\\
\CL_2(\CP_2)&=be_2n\alpha_2-b\gamma_2u=bL_2(P_2),\\
\CL_3(\CP_2)&=-a\gamma_3v+b\beta_3u+ce_3n\alpha_3.
\end{split}\]
Since
\[\frac{(-\gamma_3v)^2}{e_1}+\frac{(\beta_3u)}{e_2}+\frac{(e_3n\alpha_3)^2}{e_3}=n(-d_2\gamma_3^2+d_1\beta_3^2+e_3n\alpha_3^2)=0,\]
we have
\[\begin{split}
\CL_3(\CP_2)L_3(P_2)
=\half(a+b)(a+c)(b+c)\biggl(
\frac{e_3n\alpha_3}{a+b}+
\frac{\beta_3u}{a+c}-
\frac{\gamma_3v}{b+c}\biggr)^2
\end{split}\]
by Lemma~\ref{lem:fund-equality}.
Therefore,
\[\begin{split}
[\CL_1(\CP_2),d_1']_2&=[L_1(P_2),d_1']_2+[a,d_1']_2=[L_1(P_2),d_1']_2,\\
[\CL_2(\CP_2),d_2']_2&=[L_2(P_2),d_2']_2+[b,d_2']_2=[L_2(P_2),d_2']_2,\\
[\CL_3(\CP_2),d_3']_2&=[L_3(P_2),d_3']_2+[2(a+b)(a+c)(b+c),d_3']_2=[L_3(P_2),d_3']_2
\end{split}\]
by Lemma~\ref{lem:special-residue}.

If $e_1n+d_2\equiv e_2n-d_1\equiv 4\bmod 8$, denote by
\[\begin{aligned}
\CU&=\sqrt{(4d_3b^{-2}+e_2n)d_1},&\CV&=\sqrt{(4d_3a^{-2}-e_1n)d_2},\\
U&=\sqrt{(4d_3+e_2n)d_1},&V&=\sqrt{(4d_3-e_1n)d_2}
\end{aligned}\]
with $\CU\equiv \CV\equiv U\equiv V\equiv1\bmod4$.
Similar to (ii-a), we have $\CU\equiv U, \CV\equiv V\bmod{16}$.
Take $\CP_2=(1,b\CU/d_1,a\CV/d_2,2)$, then $P_2=(1,U/d_1,V/d_2,2)$ and
\[\begin{split}
\CL_1(\CP_2)&\equiv ae_1n\alpha_1+a\beta_1V-2d_3\gamma_1,\\
\CL_2(\CP_2)&\equiv be_2n\alpha_2+2d_3\beta_2-b\gamma_2U,\\
\CL_3(\CP_2)&\equiv -a\gamma_3V+b\beta_3U+ce_3n\alpha_3
\end{split}\]
modulo $16$.

If $\gamma_1$ is odd, then exactly one of $\alpha_1$ and $\beta_1$ is odd.
Thus $\CL_1(\CP_2)$ is odd.
If $\gamma_1$ is even, then both of $\alpha_1$ and $\beta_1$ are odd.
By choosing a suitable sign of $\alpha_1$, we may assume that $4\mid(\alpha_1+\beta_1)$.
Therefore, $2\mmid \CL_1(\CP_2)$.
Since $\CL_1(\CP_2)\equiv aL_1(P_2)\bmod8$, we have
\[[\CL_1(\CP_2),d_1']_2=[L_1(P_2),d_1']_2+[a,d_1']_2=[L_1(P_2),d_1']_2.\]
Similarly, we may assume that $2\mmid \CL_2(\CP_2)$ by choosing a suitable sign of $\alpha_2$.
Since $\CL_2(\CP_2)\equiv aL_2(P_2)\bmod8$, we have
\[[\CL_2(\CP_2),d_2']_2=[L_2(P_2),d_2']_2+[b,d_2']_2=[L_2(P_2),d_2']_2.\]

Clearly, $\beta_3$ and $\gamma_3$ are odd.
By choosing a suitable sign of $\gamma_3$, we may assume that $2\mmid \CL_3(\CP_2)$ and $2\mmid L_3(P_2)$.
Since
\[\begin{split}
&\frac14\left(\frac{(-\gamma_3V)^2}{e_1}+\frac{(\beta_3U)^2}{e_2}+\frac{(e_3n\alpha_3)^2}{e_3}\right)\\
=&d_3\left(\frac{d_1\beta_3^2}{e_2}+\frac{d_2\gamma_3^2}{e_1}\right)+\frac14n(e_3n\alpha_3^2+d_1\beta_3^2-d_2\gamma_3^2)\\
\equiv&d_3(d_1e_2+d_2e_1)
\equiv d_3\bigl((e_2n-4)e_2+(4-e_1n)e_1\bigr)\\
\equiv&4d_3(-e_2+e_1)\equiv0\bmod8
\end{split}\]
and $4\mid(e_1a+e_2b+e_3c)$, the odd number
\[\begin{split}
\frac{\CL_3(\CP_2)}2\cdot\frac{L_3(P_2)}2
\equiv&\frac18(a+b)(a+c)(b+c)\biggl(
-\frac{\gamma_3V}{b+c}
+\frac{\beta_3U}{c+a}
+\frac{e_3n\alpha_3}{a+b}\biggr)^2\bmod8
\end{split}\]
is congruent to $1$ modulo $4$ by Lemmas~\ref{lem:fund-equality} and \ref{lem:special-residue}.
Therefore
\[[\CL_3(\CP_2),d_3']_2=[L_3(P_2),d_3']_2.\]

The rest part is similar to the proof of Theorem~\ref{thm:sha-general}.
\end{proof}

\section{The even case}

Assume that $2\mmid e_1, 2\mmid e_2, 4\mid e_3$ and $E^{(n)}$ has no rational point of order $4$.
Write $e_i=2f_i$.
Let $n$ be an odd positive square-free integer.
Let $\Lambda=(d_1,d_2,d_3)$ where $d_1,d_2,d_3$ are square-free divisors of $2f_1f_2f_3n$.

\subsection{Homogeneous spaces}

Recall that $D_\Lambda^{(n)}$ is defined as
\[\begin{cases}
	H_1:& 2f_1nt^2+d_2u_2^2-d_3u_3^2=0, \\
	H_2:& 2f_2nt^2+d_3u_3^2-d_1u_1^2=0, \\
	H_3:& 2f_3nt^2+d_1u_1^2-d_2u_2^2=0
\end{cases}\]
and the $2$-torsion points of $E^{(n)}$ correspond to 
\[(1,1,1),\ (-2f_3n,-f_1f_3,2f_1n),\ (-f_2f_3,2f_3n,-2f_2n),\ (2f_2n,-2f_1n,-f_1f_2).\]
These triples have $2$-valuations $(0,0,0),(0,1,1),(1,0,1),(1,1,0)$ (not correspondingly).
Hence any element in the pure $2$-Selmer group $\Sel_2'\bigl(E^{(n)}\bigr)$ has a unique representative $\Lambda=(d_1,d_2,d_3)$ with odd $d_i\mid e_1e_2e_3n$.

\begin{lemma}\label{lem:local-solv-even}
Let $\Lambda=(d_1,d_2,d_3)$ where $d_1,d_2,d_3$ are odd.
If $D_\Lambda^{(n)}(\BQ_2)\neq\emptyset$, then $d_3\equiv1\bmod4$.
\end{lemma}
\begin{proof}
Since $v_2(t)\ge v_2(u_3)=v_2(u_2)$ by $H_1$ and
$v_2(t)\ge v_2(u_1)=v_2(u_3)$ by $H_2$, we may assume that $u_1,u_2,u_3$ are $2$-adic units and $t$ is a $2$-adic integer.
Then
\[2f_3nt^2=d_2u_2^2-d_1u_1^2\equiv d_2-d_1\bmod 8.\]
This implies that $d_2\equiv d_1\bmod 4$ and then $d_3\equiv 1\bmod 4$.
\end{proof}

\begin{lemma}\label{lem:local-solv2-even}
Let $\Lambda=(d_1,d_2,d_3)$ where $d_1,d_2,d_3$ are odd and $d_3\equiv1\bmod 4$.
If $D_\Lambda^{(n)}$ is locally solvable at all places $v\neq 2$, then $D_\Lambda^{(n)}$ is also locally solvable at $v=2$.
\end{lemma}
\begin{proof}
Similar to Lemma~\ref{lem:local-solv2-odd}, we have
\[[2f_1nd_3,d_1]_2=[2f_2nd_1,d_2]_2=[2f_3nd_2,d_3]_2=0.\]

If $(d_1,d_2,d_3)\equiv(1,1,1)\bmod 4$, then $d_1\equiv d_2\equiv d_3\equiv1\bmod 8$.
Take
\[t=0,u_1=\sqrt{d_3/d_1},\ u_2=\sqrt{d_3/d_2},\ u_3=1.\]

If $(d_1,d_2,d_3)\equiv(-1,-1,1)\bmod 4$, then $2f_1nd_3\equiv d_1+3$ and $2f_2nd_1\equiv d_2+3\bmod 8$.
In other words, $2f_1n\equiv d_2+3d_3\bmod8$ and $2f_2n\equiv d_3+3d_1\bmod 8$.
Take $t=u_3=1$, then
\[u_1^2=(d_3+2f_2n)/d_1\equiv 2d_2+3\equiv 1\bmod8\]
and
\[u_2^2=(d_3-2f_1n)/d_2\equiv -2d_1-1\equiv 1\bmod8.\]

Hence $D_\Lambda^{(n)}$ is locally solvable at $v=2$.
\end{proof}

\begin{lemma}\label{lem:local-solv-q}
Assume that $n$ is coprime with $e_1e_2e_3$.
If $q$ is an odd prime factor of $e_i$, then $D_\Lambda^{(n)}(\BQ_q)\neq\emptyset$ if and only if $q\nmid d_i$ and
\begin{itemize}
\item $\leg{d_i}q=1$, if $q\nmid d_{i+1}$;
\item $\leg{e_{i+1}nd_i}q=1$, if $q\mid d_{i+1}, q^2\nmid e_i$;
\item $\leg{e_{i+1}n}q=\leg{d_i}q=1$, if $q\mid d_{i+1}, q^2\mid e_i$.
\end{itemize}
\end{lemma}
\begin{proof}
By the symmetry, we only need to consider the case $i=1$.
Assume that $D_\Lambda(\BQ_q)\neq\emptyset$.
Since we are dealing with homogeneous spaces, we may assume that $t,u_1,u_2,u_3$ are $q$-adic integers and at least one of them is a $q$-adic unit.
If $q\mid d_1, q\mid d_2, q\nmid d_3$, then $q\mid u_3$ by $H_1$ and $q\mid t$ by $H_3$.
Therefore, $q\mid u_1$ by $H_2$ and $q\mid u_2$ by $H_3$, which is impossible.
Similarly, the case $q\mid d_1, q\nmid d_2, q\mid d_3$ is also impossible. Hence $q\nmid d_1$.

If $q\nmid d_2d_3$, then $\leg{d_1}q=\leg{d_2d_3}q=1$ by $H_1$.
Conversely, if $\leg{d_1}q=1$, then we take
\[\begin{split}
u_2&=d_1d_3/d_2,\\
u_1^2&=d_3-e_3nt^2/d_1,\\
u_3^2&=d_1+e_1nt^2/d_3\equiv d_1\bmod q,
\end{split}\]
where $t\in\BZ_q$ such that $d_3-e_3 nt^2/d_1$ is a square in $\BZ_q$.
In fact, if $e_3nd_2$ is quadratic residue modulo $q$, then we may take $t=\sqrt{\frac{d_1d_3}{e_3n}}$ and $u_1=0$;
if not, then there exists $t\in\set{0,1,\dots,(q-1)/2}$ such that $d_3-e_3nt^2/d_1\bmod q$ is a nonzero square.
Hence $D_\Lambda(\BQ_q)$ is non-empty.

If $q\mid d_2,q\mid d_3$ and $q^2\nmid e_1$, then $\leg{e_2nd_1}q=1$ by $H_2$.
Conversely, if $\leg{e_2nd_1}q=1$, then we take
\[\begin{split}
u_2^2&=d_1d_3/d_2,\\
u_1^2&=d_3-e_3nt^2/d_1\equiv e_2nt^2/d_1\bmod q,\\
u_3^2&=d_1+e_1nt^2/d_3.
\end{split}\]
Similar to the previous case, there exists $t\in\BZ_q$ such that $d_1+e_1t^2/d_3$ is a square in $\BZ_q$.
Hence $D_\Lambda(\BQ_q)$ is non-empty.

If $q\mid d_2,q\mid d_3$ and $q^2\mid e_1$, then $\leg{d_1}q=\leg{d_2d_3}q=1$ by $H_1$ and $\leg{e_2nd_1}q=1$ by $H_2$.
Conversely, if $\leg{e_2n}q=\leg{d_1}q=1$, then $\leg{-e_3nd_1}q=1$ and we take
\[\begin{split}
u_2^2&=d_1d_3/d_2,\\
u_1^2&=d_3-e_3nt^2/d_1\equiv e_2nt^2/d_1\bmod q,\\
u_3^2&=d_1+e_1nt^2/d_3\equiv d_1\bmod q.
\end{split}\]
Hence $D_\Lambda(\BQ_q)$ is non-empty.
\end{proof}

Let $\Lambda=(d_1,d_2,d_3)\in\Sel_2'\bigl(E^{(n)}\bigr)$ with odd $d_i\mid e_1e_2e_3n$ and $d_3\equiv 1\bmod 4$.
We will use the notations $\bfx,\bfy,\bfz$ in \eqref{eq:xyz}.
If $e_2>0$ and $e_3<0$, or all $p_i\equiv1\bmod4$, write
\begin{equation}\label{eq:form-even1}
\begin{split}
d_1&=p_1^{x_1}\cdots p_k^{x_k}\cdot \wt d_1,\\
d_2&=p_1^{y_1}\leg{-1}{p_1}^{z_1}\cdots p_k^{y_k}\leg{-1}{p_1}^{z_k}\cdot \wt d_2,\\
d_3&=(p_1^*)^{z_1}\cdots (p_k^*)^{z_k}\cdot \wt d_3
\end{split}
\end{equation}
where $p^*=\leg{-1}{p}p$.
Then $\wt d_1\wt d_2\wt d_3\in \BQ^{\times2}$.

\begin{theorem}\label{thm:selmer-even1}
Let $n$ be an odd positive square-free integer coprime with $e_1e_2e_3$, whose prime factors are quadratic residues modulo each odd prime factor of $e_1e_2e_3$.
Assume that 
\begin{itemize}
\item both $E$ and $E^{(n)}$ have no rational point of order $4$;
\item $e_2>0$ and $e_3<0$, or all $p_i\equiv1\bmod4$;
\item $\leg{p^*}q=1$ for any odd primes $p\mid n, q\mid e_2e_3$.
\end{itemize}
If $\Sel_2(E/\BQ)\cong(\BZ/2\BZ)^2$, then the map $(d_1,d_2,d_3)\mapsto \svec{\bfx}{\bfz}$ induces an isomorphism
\[\Sel_2'\bigl(E^{(n)}\bigr)\simto \Ker \begin{pmatrix}
\bfA+\bfD_{e_2}&\bfD_{-e_2e_3}\\
\bfD_{-e_1e_2}&\bfA^\rmT+\bfD_{e_2}
\end{pmatrix},\]
where $d_i\mid n, d_1>0, d_3\equiv 1\bmod4$.
\end{theorem}
\begin{proof}
Let $\Lambda=(d_1,d_2,d_3)$ with odd square-free $d_i\mid e_1e_2e_3n$ and denote by $\wt\Lambda=(\wt d_1,\wt d_2,\wt d_3)$.
If all $p_i\equiv1\bmod4$, then $\sgn(d_i)=\sgn(\wt d_i)$.
If $e_2>0,e_3<0$, then $\sgn(d_1)=\sgn(\wt d_1)$.
Hence $D_\Lambda^{(n)}(\BR)\neq\emptyset$ if and only if $D_{\wt\Lambda}^{(1)}(\BR)\neq\emptyset$ by Lemma~\ref{lem:local-solv-R}.

One can show that $n, d_i/\wt d_i\in\BQ_q^{\times2}$ where $q$ is an odd prime factor of $e_i$ by our assumptions.
Therefore, $D_\Lambda^{(n)}(\BQ_q)\neq\emptyset$ if and only if $D_{\wt\Lambda}^{(1)}(\BQ_q)\neq\emptyset$ by Lemma~\ref{lem:local-solv-q}.
Hence $\Lambda\in\Sel_2\bigl(E^{(n)}/\BQ\bigr)$ if and only if $\wt \Lambda\in\Sel_2(E/\BQ)$ and $D_\Lambda^{(n)}$ is locally solvable at each $p\mid n$ by Lemmas~\ref{lem:local-solv-even}, \ref{lem:local-solv2-even} and the fact $d_3\equiv\wt d_3\equiv1\bmod4$.

If $\Lambda\in\Sel_2\bigl(E^{(n)}/\BQ\bigr)$, then $\wt \Lambda\in\Sel_2(E/\BQ)$.
By our assumptions, $\wt\Lambda=(1,1,1)$.
Hence each element in $\Sel_2'\bigl(E^{(n)}\bigr)$ has a unique representative $(d_1,d_2,d_3)$ with $d_i\mid n,d_1>0,d_3\equiv1\bmod4$.
Based on this, we can express $\Sel_2'\bigl(E^{(n)}\bigr)$ in terms of linear algebra by Lemma~\ref{lem:local-solv-p} after a translation of languages.
One need the fact that 
\[\bigl([p_i^*,-n]_{p_j}\bigr)_{i,j}=\bfA^\rmT+\bfD_{-1}.\qedhere\]
\end{proof}

If $e_3>0$ and $e_1<0$, write
\[\begin{split}
d_1&=p_1^{x_1}\leg{-1}{p_1}^{z_1}\cdots p_k^{x_k}\leg{-1}{p_1}^{z_k}\cdot \wt d_1,\\
d_2&=p_1^{y_1}\cdots p_k^{y_k}\cdot \wt d_2,\\
d_3&=(p_1^*)^{z_1}\cdots (p_k^*)^{z_k}\cdot \wt d_3.
\end{split}\]
Then $\wt d_1\wt d_2\wt d_3\in\BQ^{\times2}$.
Similar to Theorem~\ref{thm:selmer-even1}, we have:

\begin{theorem}\label{thm:selmer-even2}
Let $n$ be an odd positive square-free integer coprime with $e_1e_2e_3$, whose prime factors are quadratic residues modulo each odd prime factor of $e_1e_2e_3$.
Assume that 
\begin{itemize}
\item both $E$ and $E^{(n)}$ have no rational point of order $4$;
\item $e_3>0$ and $e_1<0$;
\item $\leg{p^*}q=1$ for any odd primes $p\mid n, q\mid e_1e_3$.
\end{itemize}
If $\Sel_2(E/\BQ)\cong(\BZ/2\BZ)^2$, then the map $(d_1,d_2,d_3)\mapsto \svec{\bfy}{\bfz}$ induces an isomorphism
\[\Sel_2'\bigl(E^{(n)}\bigr)\simto \Ker \begin{pmatrix}
\bfA+\bfD_{-e_1}&\bfD_{-e_1e_3}\\
\bfD_{-e_1e_2}&\bfA^\rmT+\bfD_{-e_1}
\end{pmatrix},\]
where $d_i\mid n,d_2>0,d_3\equiv1\bmod4$.
\end{theorem}

\subsection{The Cassels pairing}

Let $(a,b,c)$ be a primitive triple of odd integers satisfying
\[e_1a^2+e_2b^2+e_3c^2=0.\]
Denote by $\CE=\mse_{e_1a^2,e_2b^2}$ and $\CE^{(n)}=\mse_{e_1a^2n,e_2b^2n}$.

\begin{theorem}\label{thm:sha-even}
Let $n$ be an odd positive square-free integer coprime with $e_1e_2e_3abc$, whose prime factors are quadratic residues modulo each odd prime factor of $e_1e_2e_3abc$.
Assume that
\begin{itemize}
\item both $E$ and $E^{(n)}$ have no rational point of order $4$;
\item if $e_2>0$ and $e_3<0$, then $q\equiv1\bmod4$ for any odd prime $q\mid e_2e_3bc$;
\item if $e_3>0$ and $e_1<0$, then $q\equiv1\bmod4$ for any odd prime $q\mid e_1e_3ac$;
\item if $e_1>0$ and $e_2<0$, then $p\equiv1\bmod4$ for any odd prime $p\mid n$.
\end{itemize}
If $\Sel_2(E/\BQ)\cong\Sel_2(\CE/\BQ)\cong(\BZ/2\BZ)^2$, then the following are equivalent:
\begin{enumerate}
\item $\rank_\BZ E^{(n)}(\BQ)=0$ and $\Sha\bigl(E^{(n)}/\BQ\bigr)[2^\infty]\cong(\BZ/2\BZ)^{2t}$;
\item $\rank_\BZ \CE^{(n)}(\BQ)=0$ and $\Sha\bigl(\CE^{(n)}/\BQ\bigr)[2^\infty]\cong(\BZ/2\BZ)^{2t}$.
\end{enumerate}
\end{theorem}

\begin{proof}
Similar to the proof of Theorem~\ref{thm:sha-general}, 
both $E^{(n)}$ and $\CE(\BQ)^{(n)}$ have no rational point of order $4$.
By choosing suitable signs, we may assume that $a\equiv b\equiv c\equiv1\bmod4$.

Assume that $e_2>0$ and $e_3<0$, or all prime factors of $n$ are congruent to $1$ modulo $4$.
Since the matrix in Theorem~\ref{thm:selmer-even1} does not depend on $a,b,c$, we have a canonical isomorphism
\[\Sel_2'\bigl(E^{(n)}\bigr)\cong\Sel_2'\bigl(\CE^{(n)}\bigr).\]
Let $\Lambda=(d_1,d_2,d_3),\Lambda'=(d_1',d_2',d_3')\in\Sel_2'\bigl(\CE^{(n)}\bigr)$ with $d_i,d_i'\mid n, d_1,d_1'>0,d_3\equiv d_3'\equiv1\bmod4$.
If $d_2<0$ and $d_2'<0$, we replace $\Lambda'$ by $\Lambda+\Lambda'$.
If $d_2>0$ and $d_2'<0$, we switch $\Lambda$ and $\Lambda'$.
Since
\[\pair{\Lambda,\Lambda'}=\pair{\Lambda,\Lambda+\Lambda'}=\pair{\Lambda',\Lambda},\]
these operations do not change $\pair{\Lambda,\Lambda'}$. 
Hence we may assume that $d_2'>0$ and $d_3'>0$.
When $e_3>0$ and $e_1<0$, we may assume that $d_1'>0$ and $d_3'>0$ similarly.

We will denote by $\CD,\CH,\CQ,\CL,\CP$ the corresponding symbols for $\CE$ and $D,H,Q,L,P$ the corresponding symbols for $E$ in the calculation of Cassels pairing.
Recall that $\CD_\Lambda^{(n)}$ is defined as
\[\begin{cases}
	\CH_1:& 2f_1a^2nt^2+d_2u_2^2-d_3u_3^2=0, \\
	\CH_2:& 2f_2b^2nt^2+d_3u_3^2-d_1u_1^2=0, \\
	\CH_3:& 2f_3c^2nt^2+d_1u_1^2-d_2u_2^2=0.
\end{cases}\]
Let $(\alpha_i,\beta_i,\gamma_i)$ be primitive triples of integers satisfying
\[\begin{split}
2f_1n\alpha_1^2+d_2\beta_1^2-d_3\gamma_1^2&=0, \\
2f_2n\alpha_2^2+d_3\beta_2^2-d_1\gamma_2^2&=0, \\
2f_3n\alpha_3^2+d_1\beta_3^2-d_2\gamma_3^2&=0.
\end{split}\]
Choose
\[\begin{aligned}
\CQ_1&=(\alpha_1,a\beta_1,a\gamma_1)\in \CH_1(\BQ),&
\CL_1&=2f_1an\alpha_1t+d_2\beta_1u_2-d_3\gamma_1u_3,\\
\CQ_2&=(\alpha_2,b\beta_2,b\gamma_2)\in \CH_2(\BQ),&
\CL_2&=2f_2bn\alpha_2t+d_3\beta_2u_3-d_1\gamma_2u_1,\\
\CQ_3&=(\alpha_3,c\beta_3,c\gamma_3)\in \CH_3(\BQ),&
\CL_3&=2f_3cn\alpha_3t+d_1\beta_3u_1-d_2\gamma_3u_2.
\end{aligned}\]

(i) The case odd $v=q\mid e_1e_2e_3abc$.
Since $\leg pq=1$ for any prime factor $p$ of $n$, $d_i'>0$ is a square modulo $q$.
Therefore, $[\CL_i(\CP_q),d_i']_q=0=[L_i(P_q),d_i']_q$.

(ii) The case $v=p\mid n$.
The proof is similar to the proof of Theorem~\ref{thm:sha-general}.

(iii) The case $v=2$. Note that $d_3\equiv1\bmod4$.

(iii-a) The case $(d_1,d_2,d_3)\equiv(1,1,1)\bmod 4$.
As shown in Lemma~\ref{lem:local-solv2-even}, we have $d_1\equiv d_2\equiv d_3\equiv 1\bmod8$, 
take $\CP_2=(0,1/\sqrt{d_1},1/\sqrt{d_2},1/\sqrt{d_3})=P_2$.
Then
\[\begin{split}
\CL_1(\CP_2)&=\beta_1\sqrt{d_2}-\gamma_1\sqrt{d_3}=L_1(P_2),\\
\CL_2(\CP_2)&=\beta_2\sqrt{d_3}-\gamma_2\sqrt{d_1}=L_2(P_2),\\
\CL_3(\CP_2)&=\beta_3\sqrt{d_1}-\gamma_3\sqrt{d_2}=L_3(P_2).
\end{split}\]

(iii-b) The case $(d_1,d_2,d_3)\equiv(-1,-1,1)\bmod 4$.
As shown in Lemma~\ref{lem:local-solv2-even}, we have $(d_3+2f_2b^2n)d_1\equiv (d_3-2f_1a^2n)d_2\equiv 1\bmod 8$.
Denote by 
\[\begin{aligned}
\CU&=\sqrt{(d_3+2f_2b^2n)d_1}, &\CV&=\sqrt{(d_3-2f_1a^2n)d_2},\\
U&=\sqrt{(d_3+2f_2n)d_1}, &V&=\sqrt{(d_3-2f_1n)d_2} 
\end{aligned}\]
with $\CU\equiv \CV\equiv U\equiv V\equiv 1\bmod 4$.
Since $\CU^2\equiv U^2\bmod{16}$, we have $\CU\equiv U\bmod 8$.
Similarly, $\CV\equiv V\bmod8$.

Take $\CP_2=(1,\CU/d_1,\CV/d_2,1)$, then $P_2=(1,U/d_1,V/d_2,1)$.
Note that all $\beta_i,\gamma_i$ are odd.
By choosing suitable signs of $\gamma_i$, we may assume that $2\mmid \CL_i(\CP_2)$.
Since
\[\begin{split}
\CL_1(\CP_2)&\equiv2f_1an\alpha_1+\beta_1V-d_3\gamma_1\equiv L_1(P_2),\\
\CL_2(\CP_2)&\equiv2f_2bn\alpha_2+d_3\beta_2-\gamma_2U\equiv L_2(P_2),\\
\CL_3(\CP_2)&\equiv2f_3cn\alpha_3+\beta_3U-\gamma_3V\equiv L_3(P_2)
\end{split}\]
modulo $8$, we have
\[[\CL_i(\CP_2),d_i']_2=[L_i(P_2),d_i']_2.\]

The rest part is similar to the proof of Theorem~\ref{thm:sha-general}.
\end{proof}

\section{Congruent elliptic curves}

Assume that $n=p_1\cdots p_k\equiv1\bmod 4$.
Denote by
\[h_{2^s}(n)=\dim_{\BF_2}\frac{2^{s-1}\Cl\bigl(\BQ(\sqrt{-n})\bigr)}{2^s\Cl\bigl(\BQ(\sqrt{-n})\bigr)}\]
the $2^s$-rank of the class group of $\BQ(\sqrt{-n})$.
By Gauss genus theory and R\'edei's work in \cite{Redei1934}, we can characterize $h_2(n)$ and $h_4(n)$.
See \cite[\S~3]{Wang2016} for more details.

\begin{proposition}\label{pro:4rank}
We have $h_2(n)=k$ and $h_4(n)=k-\rank(\bfA,\bfD_2{\bf1})$.
\end{proposition}

Denote by
\[E=\mse_{1,1}: y^2=x(x-1)(x+1)\]
the congruent elliptic curve and $E^{(n)}=\mse_{n,n}$.
Let $(a,b,c)$ be a primitive triple of positive integers satisfying $a^2+b^2=2c^2$.
Then $a,b,c$ are odd.
Denote by $\CE=\mse_{a^2,b^2},\CE^{(n)}=\mse_{a^n,b^2n}$.

\begin{theorem}[{\cite[Theorem~4.4]{WangZhang2022}}]\label{thm:cong-odd-2}
Let $n\equiv 1\bmod8$ be an positive square-free integer coprime with $abc$, where each prime factor of $n$ is a quadratic residue modulo every odd prime factor of $abc$.
Assume that 
\begin{itemize}
\item $p\equiv1\bmod 4$ for all primes $p\mid n$;
\item $\Sel_2(\CE/\BQ)\cong(\BZ/2\BZ)^2$.
\end{itemize}
Then the following are equivalent:
\begin{enumerate}
\item $\rank_\BZ \CE^{(n)}(\BQ)=0$ and $\Sha\bigl(\CE^{(n)}/\BQ\bigr)[2^\infty]\cong(\BZ/2\BZ)^2$;
\item $h_4(n)=1$ and $h_8(n)\equiv\frac{d-1}4\bmod2$.
\end{enumerate}
Here $d\neq 1,n$ is a positive factor of $n$ such that $(d,-n)_v=1,\forall v$, or $(2d,-n)_v=1,\forall v$.
\end{theorem}
\begin{proof}
Since $\Sel_2(E/\BQ)\cong(\BZ/2\BZ)^2$, this result follows from Theorem~\ref{thm:sha-odd} and \cite[Theorem~1.1]{Wang2016} directly.
\end{proof}

\begin{theorem}\label{thm:cong-even-2}
Let $n\equiv 1\bmod 8$ be a positive square-free integer coprime with $abc$, where each prime factor of $n$ is a quadratic residue modulo every prime factor of $abc$.
Assume that
\begin{itemize}
\item either $n$ or $a$ or $b$ has no prime factor $\equiv 3\bmod4$;
\item $p\equiv\pm1\bmod 8$ for all primes $p\mid n$;
\item $\Sel_2\bigl(\CE^{(2)}/\BQ\bigr)\cong(\BZ/2\BZ)^2$.
\end{itemize}
Then the following are equivalent:
\begin{enumerate}
\item $\rank_\BZ \CE^{(2n)}(\BQ)=0$ and $\Sha\bigl(\CE^{(2n)}/\BQ\bigr)[2^\infty]\cong(\BZ/2\BZ)^2$;
\item $h_4(n)=1$ and $d\equiv 9\bmod 16$.
\end{enumerate}
Here, $d$ is the unique divisor of $n$ such that $d\neq1,d\equiv1\bmod4$ and $(d,n)_v=1,\forall v$.
\end{theorem}

\begin{proof}
For any prime $q\mid c$, we have $a^2\equiv-b^2\bmod q$.
Therefore $q\equiv1\bmod 4$ and $\leg{p^*}q=\leg pq=1$.
If $n$ or $b$ has no prime factor $\equiv 3\bmod4$, then $\leg{p^*}q=\leg pq=1$ for all primes $p\mid n, q\mid b$.
We apply Theorem~\ref{thm:selmer-even1} to $(e_1,e_2,e_3)=(2a^2,2b^2,-4c^2)$, the map $(d_1,d_2,d_3)\mapsto \svec{\bfx}{\bfz}$ induces an isomorphism
\[\Sel_2'\bigl(\CE^{(n)}\bigr)\simto \Ker \bfM 
\quad\text{where}\quad
\bfM=\begin{pmatrix}
\bfA+\bfD_2&\bfD_2\\
\bfD_{-1}&\bfA^\rmT+\bfD_2
\end{pmatrix}=\begin{pmatrix}
\bfA&\\
\bfD_{-1}&\bfA^\rmT
\end{pmatrix}\]
and $d_i\mid n, d_1>0, d_3\equiv 1\bmod4$.

One can show that
\[\Ker\bfM\supseteq\set{\svec{\bf0}{\bfd},
\svec{\bf1}{\bfd+\bf1}: \bfd\in\Ker\bfA^\rmT}.\]
Since $\bfA{\bf1}={\bf0}$, we have $\rank\bfA^\rmT=\rank\bfA\le k-1$ and then $\Ker\bfM$ has at least four vectors.
Hence
\[\dim_{\BF_2}\Sel_2'\bigl(\CE^{(n)}\bigr)=2\iff \rank\bfA=k-1\iff	h_4(n)=1\]
by Proposition~\ref{pro:4rank}.

Assume that $h_4(n)=1$.
Note that $(p_j,-n)_{p_i}=(p_i^*,n)_{p_j}$.
Therefore, $\bfA^\rmT\bfd=0$ if and only if $(d,n)_p=1$ for all $p\mid n$, where $d=(p_1^*)^{s_1}\cdots (p_k^*)^{s_k}$, $\bfd=(s_1,\dots,s_k)^\rmT$.
Hence $\Sel_2'\bigl(\CE^{(n)}\bigr)$ is generated by $\Lambda=(n,1,n)$ and $\Lambda'=(1,d,d)$.

By Theorem~\ref{thm:sha-even}, we may assume that $a=b=c=1$.
Recall that $D_\Lambda^{(n)}$ is defined as
\[\begin{cases}
H_1:& 2nt^2+u_2^2-nu_3^2=0,\\
H_2:& 2t^2+u_3^2-u_1^2=0,\\
H_3:& -4nt^2+nu_1^2-u_2^2=0.
\end{cases}\]
Choose 
\begin{align*}
Q_2&=(0,1,1)\in H_2(\BQ),& L_2&=u_1-u_3,\\
Q_3&=(1,0,-2)\in H_3(\BQ),& L_3&=2t+u_1.
\end{align*}
By Lemma~\ref{lem:cassels}, we have
\[\pair{\Lambda,\Lambda'}_{E^{(n)}}=\sum_{v\mid 2n\infty} \bigl[L_2L_3(P_v),d\bigr]_v\]
for any $P_v\in D_\Lambda^{(n)}(\BQ_v)$.

For $v\mid n\infty$, take $P_v=(1,2,0,-\sqrt2)$, then $L_2L_3(P_v)=4(2+\sqrt2)$ and $\pair{\Lambda,\Lambda'}_v=[2+\sqrt2,d]_v$.
For $v=2$, take $P_2=(0,1,\sqrt n,-1)$.
Then $L_2L_3(P_2)=2$ and $\pair{\Lambda,\Lambda'}_2=[2,d]_2=0$.
Hence $\pair{\Lambda,\Lambda'}_{E^{(n)}}=\aleg{2+\sqrt2}{|d|}\equiv \frac{d-1}8\bmod 2$ by Lemma~\ref{lem:special-symbol}.
Conclude the results by Lemma~\ref{lem:non-deg}.

If $a$ has no prime factor $\equiv 3\bmod4$, then $\leg{p^*}q=\leg pq=1$ for all primes $p\mid n, q\mid a$.
We apply Theorem~\ref{thm:selmer-even2} to $(e_1,e_2,e_3)=(-2b^2,-2a^2,4c^2)$.
Then we can prove the result similarly.
\end{proof}

\begin{lemma}\label{lem:special-symbol}
Let $m\equiv1\bmod8$ be a square-free integer with prime factors congruent to $\pm1$ modulo $8$. 
Then $m\equiv 1\bmod{16}$ if and only if $\leg{2+\sqrt2}{|m|}=1$.
\end{lemma}
\begin{proof}
Write $m=u^2-2w^2\equiv 1\bmod 8$.
Denote by $\mu=u+w$ and $\lambda=u+2w$.
Then $m=2\mu^2-\lambda^2$ and $u,\mu,\lambda$ are odd.
Let $w'$ be the positive odd part of $w$.
Then
\[\legbigg{w}{|m|}=\legbigg{m}{w'}=\legbigg{u^2-2w^2}{w'}=1, \]
\[\legbigg{\lambda}{|m|}=\legbigg{m}{|\lambda|}=\legbigg{2\mu^2-\lambda^2}{\lambda}=\legbigg{2}{|\lambda|}\]
and $\lambda=u+2w\equiv(2\pm\sqrt2)w\bmod m$.
Hence
\[\legbigg{2+\sqrt2}{|m|}=\legbigg{2}{|\lambda|}.\]
Since $m+\lambda^2=2\mu^2\equiv 2\bmod 16$, we have 
\[m\equiv1\bmod 16\iff \lambda\equiv \pm1\bmod8\iff\legbigg{2}{|\lambda|}=1\iff\legbigg{2+\sqrt2}{|m|}=1.\qedhere\]
\end{proof}

\section*{Acknowledgements}
The author would like to thank Zhangjie Wang for his proof of Lemma~\ref{lem:local-solv-q}.
The author is partially supported by National Natural Science Foundation of China (Grant No. 12001510).

%

\begin{thebibliography}{Wan16}

\bibitem[Cas98]{Cassels1998}
J.~W.~S. Cassels.
\newblock Second descents for elliptic curves.
\newblock {\em J. Reine Angew. Math.}, 494:101--127, 1998.
\newblock Dedicated to Martin Kneser on the occasion of his 70th birthday.

\bibitem[Chi20]{Chiu2020}
Ching-Heng Chiu.
\newblock Strong {S}elmer companion elliptic curves.
\newblock {\em J. Number Theory}, 217:376--421, 2020.

\bibitem[Kis04]{Kisilevsky2004}
H.~Kisilevsky.
\newblock Rank determines semi-stable conductor.
\newblock {\em J. Number Theory}, 104(2):279--286, 2004.

\bibitem[MR15]{MazurRubin2015}
Barry Mazur and Karl Rubin.
\newblock Selmer companion curves.
\newblock {\em Trans. Amer. Math. Soc.}, 367(1):401--421, 2015.

\bibitem[Ono96]{Ono1996}
Ken Ono.
\newblock Euler's concordant forms.
\newblock {\em Acta Arith.}, 78(2):101--123, 1996.

\bibitem[PZ89]{ParshinZarhin1989}
Alexey~N. Papikian and Yuri~G. Zarhin.
\newblock Finiteness problems in {D}iophantine geometry.
\newblock {\em Amer. Math. Soc. Transl.}, 143:35--102, 1989.

\bibitem[Rei34]{Redei1934}
L.~R\'{e}dei.
\newblock Arithmetischer {B}eweis des {S}atzes \"{u}ber die {A}nzahl der durch
  vier teilbaren {I}nvarianten der absoluten {K}lassengruppe im quadratischen
  {Z}ahlk\"{o}rper.
\newblock {\em J. Reine Angew. Math.}, 171:55--60, 1934.

\bibitem[Sil09]{Silverman2009}
Joseph~H. Silverman.
\newblock {\em The arithmetic of elliptic curves}, volume 106 of {\em Graduate
  Texts in Mathematics}.
\newblock Springer, Dordrecht, second edition, 2009.

\bibitem[Wan16]{Wang2016}
Zhang~Jie Wang.
\newblock Congruent elliptic curves with non-trivial {S}hafarevich-{T}ate
  groups.
\newblock {\em Sci. China Math.}, 59(11):2145--2166, 2016.

\bibitem[WZ22]{WangZhang2022}
Zhangjie Wang and Shenxing Zhang.
\newblock On the quadratic twist of elliptic curves with full $2$-torsion.
\newblock {\em preprint}, 2022.

\bibitem[Yu19]{Yu2019}
Myungjun Yu.
\newblock 2-{S}elmer near-companion curves.
\newblock {\em Trans. Amer. Math. Soc.}, 372(1):425--440, 2019.

\end{thebibliography}

\end{document}